\documentclass[11pt,letterpaper,reqno]{amsart}
\usepackage{amsmath}
\usepackage{amstext}
\usepackage{amssymb}
\usepackage{amsfonts}
\usepackage{enumerate}
\usepackage{amsthm}
\usepackage{mathrsfs}
\usepackage[all]{xy}
\usepackage{color}

\numberwithin{equation}{section}

\newtheorem{theorem}{Theorem}[section]
\newtheorem{lemma}[theorem]{Lemma}
\newtheorem{proposition}[theorem]{Proposition}

\theoremstyle{definition}
\newtheorem{remark}[theorem]{Remark}
\newtheorem{definition}[theorem]{Definition}

\DeclareMathOperator{\ran}{ran}

\DeclareMathOperator{\Aff}{Aff}

\newcommand{\n}[1]{ \left\|#1\right\| }
\newcommand{\bign}[1]{ \big\|#1\big\| }

\newcommand{\N}{{\mathbb{N}}}
\newcommand{\R}{{\mathbb{R}}}

\newcommand{\eps}{{\varepsilon}}

\newcommand{\st}{{\; : \; }}

\renewcommand*{\L}{{\mathcal{L}}} 

\newcommand{\Lr}{{\L^r}}

\newcommand\restr[2]{{
  \left.\kern-\nulldelimiterspace 
  #1 
  \right|_{#2} 
  }}

\author[J.A.~Ch\'avez-Dom\'{\i}nguez]{Javier Alejandro Ch\'avez-Dom\'{\i}nguez}
\address{Department of Mathematics,
University of Oklahoma,
Norman OK , 73019-3103 USA}

\email{jachavezd@ou.edu}

\subjclass[2010]{Primary: 46A22, Secondary: 46B42, 46A32}

\keywords{Banach lattices, Regular maps, Liftings, Positive Approximation Property}

\title{Ando-Choi-Effros liftings for regular maps between Banach lattices}

\begin{document}

\begin{abstract}
The Ando-Choi-Effros lifting theorem provides conditions under which a bounded linear mapping taking values in a quotient space can be lifted through the quotient map.
We prove two versions of said theorem for regular maps between Banach lattices.
Our conditions mirror the classical ones, but additionally taking into account the order structure.
\end{abstract}

\maketitle

\section{Introduction}

Given a bounded linear map $T : Y \to X/J$ taking values in a quotient space, it is not always possible to find a lifting of $T$ to $X$: that is, a bounded linear map $L : Y \to X$ such that $q \circ  L = T$ where $q : X \to X/J$ is the canonical quotient map.
The classical Ando-Choi-Effros theorem \cite{Ando-nonempty,Choi-Effros} provides conditions under which such a lifting is guaranteed to exist.
The first condition is for the space $Y$ to be separable and have the bounded approximation property:
recall that a Banach space is said to have the \emph{approximation property} (AP) if the identity operator can be uniformly approximated on compact subsets by operators of finite rank, and is said to have the \emph{$\lambda$-bounded approximation property} ($\lambda$-BAP) if these finite-rank operators can be chosen to have norm at most $\lambda$.
The second condition is geometric in nature, and requires $J$ to sit inside $X$ is a particular way which generalizes how a closed two-sided ideal sits inside a $C^*$-algebra; the technical name is that $J$ is an \emph{$M$-ideal} in $X$. 

The Ando-Choi-Effros theorem is an abstract generalization of several well-known extension results, and thus has them as corollaries, including those of
Borsuk-Dugundji \cite{Borsuk,Dugundji}, 
and Michael-Pe{\l}czy{\'n}ski \cite{Michael-Pelczynski}.
More recently, the Ando-Choi-Effros theorem has played a recurrent role in the study of approximation properties for Lipschitz-free spaces, see for example \cite{Godefroy-Ozawa,Godefroy,BorelMathurin,CD-ACE-respecting} and the survey \cite{Godefroy-survey}.
The reader is referred to  \cite[Sec. II.6]{Harmand-Werner-Werner} for a more detailed historical account of the results that both preceded and followed the Ando-Choi-Effros theorem.

The first of the two main results of this paper (Theorem \ref{thm-Ando-Choi-Effros-regular maps}) is a lifting theorem in the Ando-Choi-Effros vein in the context of Banach lattices, where the initial map $T$ is regular
(i.e. it is a difference of two positive maps)
 and the lifting can be chosen to be regular as well.
Our conditions mirror the classical ones:
$Y$ will be assumed to have the $\lambda$-bounded positive approximation property ($\lambda$-BPAP), a version of the $\lambda$-BAP where the finite-rank approximations are taken to be positive operators, and the $M$-ideal condition is 
similarly replaced by a version that ``plays well with the order'' (see Definition \ref{def-lattice-M-ideal}). 
 
Lifting theorems in the Ando-Choi-Effros style for Banach spaces with an order have been proved by Ando \cite[Thm. 6]{Ando-closedrange}
and Vesterstr\o{}m \cite[Thm. 9]{Vestertrom},
though our approach is different in several ways.
First, the approximation properties considered by both Ando and Vesterstr\o{}m require the finite-rank approximations to be projections, whereas ours do not.
We require only the bounded positive approximation property, which Vesterstr\o{}m mentions as desirable in \cite[Remark, p. 210]{Vestertrom}.
It should be mentioned that a related result requiring this weaker hypothesis of the bounded positive approximation property 
 was achieved shortly thereafter by Andersen \cite[Thm. 5]{Andersen}.
However, Andersen's conditions differ from ours in another sense: 
though stated in a different way, in the language of Banach lattices his result requires the existence of a strong order unit.
Moreover, the main difference between our results and the previous ones is the fact that we are considering regular maps $T$ and obtaining extensions whose regular norm is controlled. 
As far as we can tell, this is the first time that Ando-Choi-Effros liftings have been considered for regular maps.

There is a second version of the Ando-Choi-Effros Theorem, where the approximation properties on the domain space $Y$ are replaced by the condition that $J$ is an $L_1$-predual.
In the proof, the key property of $L_1$-preduals that is used is the fact that their biduals are injective Banach spaces.
In the Banach lattice setting we prove a corresponding result on Ando-Choi-Effros extensions for regular maps (Theorem \ref{thm-Ando-Choi-Effros-regular maps-injective}), where the domain space is assumed to be a Banach lattice whose bidual is injective as a Banach lattice; such lattices have been characterized by Cartwright  \cite{Cartwright}.

Our proofs are inspired by the typical techniques used for proving results of this sort, which can be described as a careful process of gluing together finite-dimensional pieces.
This will cause significant issues for us, since in a Banach lattice a finite-dimensional subspace is not always contained in a finite-dimensional sublattice.
Therefore we will be requiring suitable technical conditions (namely, that the lattice be Dedekind complete), so that finite-dimensional subspaces will be guaranteed to `almost' be contained in finite-dimensional sublattices (see Lemma \ref{lemma-LO}).   

The rest of this paper is organized as follows.
In Section \ref{sec-notation} we introduce notation and some preliminary results on Banach spaces and lattices.
In Section \ref{sec-order-M-ideals} we recall a notion of $M$-ideal well-suited for the lattice setting.
The short Section \ref{sec-LR-a-la-Dean} presents a version of the Principle of Local Reflexivity for lattices in the style of \cite{Dean73}.
Section \ref{sec-preparatory} contains various technical results that will be used in the proof of the main results, which are proved in Sections \ref{sec-main} and \ref{sec-main-2}.

\section{Notation and preliminaries}\label{sec-notation}

We will consider only real normed spaces.
The (topological) dual of a normed space $X$ will be denoted by $X^*$.
Recall that if $X$ is a Banach space, a linear projection $P : X \to X$ is called an \emph{$M$-projection} (resp. \emph{$L$-projection}) if for all $x \in X$ we have $\n{x} = \max\{ \n{Px}, \n{x-Px} \} $ (resp. $ \n{x} = \n{Px} + \n{x-Px}$).
A closed subspace $J \subset X$ is called an \emph{$M$-summand} (resp. \emph{$L$-summand}) if it is the range of an $M$-projection (resp. $L$-projection),
and it is called an \emph{$M$-ideal} if $J^\perp$ is an $L$-summand in $X^*$.
For the general theory of $M$-ideals in Banach spaces, we refer the reader to the monograph \cite{Harmand-Werner-Werner}.

We will use standard notation for vector and Banach lattices and their theory, as in the books \cite{Aliprantis-Burkinshaw,Meyer-Nieberg,Schaefer}.
Given an ordered vector space $X$, we write $X_+$ for its positive cone.
By a \emph{vector sublattice} of a vector lattice $X$ we mean a linear subspace of $X$ closed under the lattice operations.
An \emph{order ideal} of $X$ is a vector sublattice $Y$ that additionally is \emph{solid}, that is, whenever $y \in Y$ and $|x| \le y$, we have $x \in Y$.
A vector lattice is called \emph{Dedekind complete} if every nonempty subset bounded above has a supremum.

Let $X$ and $Y$ be Banach lattices.
An linear operator $T : X \to Y$ is \emph{positive} when $Tx \ge 0$ for each $x \in X_+$.
We denote the set of positive operators between $X$ and $Y$ by $\L_+(X,Y)$.
We write $\Lr(X,Y)$  for the Banach space of all \emph{regular operators} (i.e., operators that can be written as the difference of two positive linear maps) from $X$ to $Y$, endowed with the norm
$$
\n{T}_r = \inf \big\{ \n{S} \; : \; S \in \L_+(X,Y) \text{ such that } \pm T \le S  \big\}.
$$
Note that this makes sense more generally for normed vector lattices, not necessarily complete.
In general, $\Lr(X,Y)$ need not be a lattice.
If $\sup\{ |Ty| : |y| \le x\}$ exists for every $x \in X_+$ then $|T|$ exists and, for every $x \in X_+$, one has that $|T|(x) = \sup \{ |Ty| : |y| \le x \}$.
This is the case, for instance, if $T$ is finite-rank or if $Y$ is Dedekind complete. If $Y$ is Dedekind complete then $\Lr(X,Y)$ is a Banach lattice and $\n{T}_r = \n{ \, |T|\, }$ for each $T \in \Lr(X,Y)$.
Since $X^* = \Lr(X,\R)$ and $\R$ is Dedekind complete, the dual of a Banach lattice $X$ is always a Dedekind complete Banach lattice.
A linear operator $T : X \to Y$ is called \emph{almost interval preserving} if for every $x \in X_+$ we have that $T[0,x]$ is dense in $[0,Tx]$.

If $E$ and $F$ are Banach lattices, the \emph{projective cone} $C_p$ in $E \otimes F$ is
$$
C_p = \Big\{  \sum_{j=1}^n x_j \otimes y_j  \; : \; x_j \in E_+, y_j \in F_+  \Big\}
$$ 

The \emph{Fremlin tensor product} (also known as \emph{positive projective tensor product}) \cite{Fremlin,Labuschagne} of $E$ and $F$, denoted $E \otimes_{|\pi|} F$,
is the completion of the algebraic tensor product $E \otimes F$ with respect to the norm
$$
\n{u}_{|\pi|} = \inf\Big\{  \sum_{j=1}^n \n{x_j}\cdot\n{y_j} \; : \;  x_j \in E_+, y_j \in F_+, \sum_{j=1}^n x_j \otimes y_j \pm u \in C_p  \Big\}.
$$
This space, with the order given by the closure of $C_p$ with respect to $\n{\cdot}_{|\pi|}$ becomes a Banach lattice.
Moreover, the dual of $E \otimes_{|\pi|}F$ can be canonically identified with $\L^r(E,F^*)$:
the mapping $\Psi : \L^r(E,F^*) \to (E \otimes_{|\pi|}F)^*$ defined by $(\Psi T)(x \otimes y) = (Tx)(y)$ for $x\in E$ and $y\in F$ is an isometric isomorphism of Banach lattices.

Notice that if $T_1 \in \L_+(E_1,F_1)$ and $T_2 \in \L_+(E_2,F_2)$, then {$T_1 \otimes T_2  \in \L_+( E_1 \otimes_{|\pi|} E_2, F_1 \otimes_{|\pi|} F_2)$} and additionally $\n{T_1 \otimes T_2 } \le \n{T_1} \cdot \n{T_2}$.
Moreover, the ``projectivity'' of the $|\pi|$ tensor norm means the following \cite[Thm. 5.2]{Labuschagne}:
If $E_0 \subset E$, $F$ are Banach lattices and $q : E \to E_0$ is almost interval preserving and a metric surjection, then so is $q\otimes Id_F: E \otimes_{|\pi|} F \to E_0 \otimes_{|\pi|} F$.

If $X$ is a Banach lattice, recall that a projection $P: X \to X$ is an \emph{order projection} (that is, the positive projection associated to a projection band) if and only if
$0 \le P \le Id_X$, see \cite[Thm. 1.44]{Aliprantis-Burkinshaw}.

A Banach lattice $X$ is said to have the \emph{$\lambda$-bounded positive approximation property} ($\lambda$-BPAP, for short) if for every finite set $K \subset X$ and every $\eps>0$ there exists a finite-rank positive operator $T :X \to X$ such that $\n{T} \le \lambda$ and $\n{x-Tx} \le \eps$ for all $x \in K$.
By standard arguments, this is equivalent to the following:
for every finite-dimensional subspace $E$ of $X$ and every $\eps>0$, there exists a finite-rank positive operator $T :X \to X$ such that $\n{T} \le \lambda$ and $\n{ \restr{(T-Id)}{E} } \le \eps$.
See \cite{Blanco} for the closely related notion of the  \emph{$\lambda$-bounded lattice approximation property}, where the norm requirement on the map $T$ is replaced by $\n{T}_r \le \lambda$. 

We will call a Banach space $X$  \emph{injective} if whenever $Y_0 \subseteq Y$ are Banach spaces and $T : Y_0 \to X$ is a bounded linear operator, there exists an extension $\tilde{T} : Y \to X$ with $\bign{\tilde{T}} = \n{T}$ (this notion is sometimes called $1$-injective).
A Banach lattice will be called injective if it satisfies the same property, but with the maps $T$, $\tilde{T}$ being positive (and $Y_0$, $Y$ being Banach lattices).

We write $X \cong Y$ to indicate that $X$ and $Y$ are isometrically isomorphic Banach spaces,
and $X \equiv Y$ to indicate that there exists an isometric lattice isomorphism from $X$ onto $Y$.

\section{Order $M$-ideals}\label{sec-order-M-ideals}

In order to prove our version of the Ando-Choi-Effros theorem for Banach lattices, we need a notion of $M$-ideals specific to this setting.
If $X$ is a Banach lattice, an \emph{order $M$-projection} (resp. \emph{order $L$-projection}) is an $M$-projection (resp. $L$-projection) which is also an order projection.
Essentially the same concept has been already considered by Haydon \cite[Def. 3A]{Haydon}
and by Ando \cite[Sec. 3]{Ando-closedrange} (who uses the term \emph{hypostrict}).
Thus, we define:

\begin{definition}\label{def-lattice-M-ideal}
Let $X$ be a Banach lattice. A closed subspace $J \subseteq X$ is called an \emph{order $M$-ideal}
if $J^\perp$ is the range of an order $L$-projection on $X^*$.
\end{definition}

Note that, in particular, an order $M$-ideal is an $M$-ideal.
The following theorem shows that an order ideal which is also an $M$-ideal is automatically an order $M$-ideal.

\begin{theorem}\label{thm-order-M-ideal}
Let $X$ be a Banach lattice and let $J \subseteq X$ be both an order ideal and an $M$-ideal.
Then the $L$-projection $P$ from $X^*$ onto $J^\perp$ is an order projection.
\end{theorem}

\begin{proof}
Define
$$
J^\# = \{ x^* \in X^* \st \n{x^*} = \bign{ \restr{x^*}{J} } \}.
$$
From the proof of \cite[Thm. I.2.2]{Harmand-Werner-Werner} it follows that each $x^* \in X^*$ can be written in a unique way as $x^* = u^* + v^*$ with $u^* \in J^\perp$ and $v^* \in J^\#$, and moreover the map $P_1 : x^* \mapsto v^*$ is an $L$-projection from $X^*$ onto $J^\#$, with kernel $J^\perp$.

Since $J$ is an order ideal, $J^\perp$ is a band in the Dedekind complete Banach lattice $X^*$.
Therefore, $X^* = J^\perp \oplus (J^\perp)^d$ holds, and in particular there exists a contractive projection $Q$ from $X^*$ onto $(J^\perp)^d$ with kernel $J^\perp$, which in fact satisfies $0 \le Q \le Id_{X^*}$. By \cite[Prop. I.1.2]{Harmand-Werner-Werner} we have that $Q = P_1$.
It follows that $P=Id_{X^*} - P_1$ satisfies  $0 \le P \le Id_{X^*}$.
\end{proof}

The following example of an order $M$-ideal corresponds to one of the most important classical examples of $M$-ideals.

\begin{proposition}\label{prop-c0-is-order-M-ideal}
Let $X$ be a Banach lattice, and $(E_n)$ a sequence of sublattices of $X$.
Then $c_0(E_n)$ is an order $M$-ideal in $c(E_n)$.
\end{proposition}

\begin{proof}
It is easy to see that $c_0(E_n)$ is an order ideal in $c(E_n)$.
Moreover, it is well-known that $c_0(E_n)$ is an $M$-ideal in $c(E_n)$:
for example, see the proof of  \cite[Prop. II.2.3]{Harmand-Werner-Werner}.
Now Theorem \ref{thm-order-M-ideal} gives the desired result.
\end{proof}

The following theorem shows that an order $M$-projection on a Banach lattice $Y$ automatically induces an $M$-projection on a space of $Y$-valued regular operators. Compare to \cite[Lemma VI.1.1]{Harmand-Werner-Werner}.

\begin{theorem}\label{thm-M-projection-on-Lr}
Let $X$ and $Y$ be Banach lattices, and suppose $P : Y \to Y$ is an $M$-projection such that $0 \le P \le Id_Y$.
Then $\tilde{P} : S \mapsto P \circ S$ is an $M$-projection on $\Lr(X,Y)$.
\end{theorem}

\begin{proof}
Let $U,V \in \Lr(X,Y)$. Given $\varepsilon>0$, there exist $U_\varepsilon, V_\varepsilon \in \L(X,Y)_+$ such that $\n{U_\varepsilon} \le \n{U}_r+\varepsilon$, 
$\n{V_\varepsilon} \le \n{V}_r+\varepsilon$, and for each $x \in X$,
$$
|Ux| \le U_\varepsilon|x|, \quad |Vx| \le V_\varepsilon|x|.
$$
Therefore, for each $x \in X$ we have
\begin{multline*}
|PVx + (Id_Y-P)Ux| \le |PVx| + |(Id_Y-P)Ux| \le P|Vx| + (Id_Y - P)|Ux| \\
\le PV_\varepsilon|x| + (Id_Y-P)U_\varepsilon|x| = \big[ PV_\varepsilon + (Id_Y-P)U_\varepsilon \big] |x|.
\end{multline*}
Note that $PV_\varepsilon + (Id_Y-P)U_\varepsilon \in \L(X,Y)_+$ and, since $S \mapsto P \circ S$ is an $M$-projection on $\L(X,Y)$ \cite[Lemma VI.1.1]{Harmand-Werner-Werner},
$$
\n{PV_\varepsilon + (Id_Y-P)U_\varepsilon} \le \max\{ \n{V_\varepsilon}, \n{U_\varepsilon} \} \le \max\{ \n{V}_r, \n{U}_r \} + \varepsilon,
$$
and therefore
$$
\big\| \tilde{P}V + (Id_Y-\tilde{P})U \big\|_r \le  \max\{ \n{V}_r, \n{U}_r \}.
$$
By the remark on \cite[p. 2]{Harmand-Werner-Werner}, this implies that $\tilde{P}$ is an $M$-projection on $\Lr(X,Y)$.
\end{proof}

\begin{remark}
If in addition $Y$ is Dedekind complete, in which case $\Lr(X,Y)$ is a Banach lattice, then the projection $\tilde{P}$ is in fact an order $M$-projection on $\Lr(X,Y)$; see 
\cite[Exercise 1.3.11.(a)]{Aliprantis-Burkinshaw}.
\end{remark}

\section{Local reflexivity for lattices, \`a la Dean}\label{sec-LR-a-la-Dean}

Dean's version of the Principle of Local Reflexivity \cite{Dean73} asserts that when $E$ and $X$ are Banach spaces with $E$ finite-dimensional, then $\L(E;X)^{**} \cong \L(E;X^{**})$ with the identification given by
\begin{equation}\label{eqn-dean-identification}
\varphi \mapsto \big[ \tilde{\varphi} :   e \mapsto \big[ x^* \mapsto   \varphi( e \otimes x^* ) \big] \big].
\end{equation}

We will need a version of this identity for Banach lattices, which is probably folklore but we have been unable to find it explicitly stated in the literature.

\begin{proposition}\label{prop-PLRL-a-la-Dean}
Let $E$ and $X$ be Banach lattices and suppose $E$ is finite-dimensional. Then  $\Lr(E,X)^{**} \equiv \Lr(E,X^{**})$, with the identification given by \eqref{eqn-dean-identification}
\end{proposition}

\begin{proof}
From  \cite[Lemma 3.4]{Blanco} we have that $\Lr(E,X)^* \equiv E \otimes_{|\pi|} X^*$.
Taking the dual,  the desired result follows from the basic properties of the projective Fremlin tensor product (i.e. identifying its dual with a space of regular operators, see Section \ref{sec-notation}).
\end{proof}

\section{Some preparatory results}\label{sec-preparatory}

In this section we collect various preparatory technical results that will be used in the prof of our first main theorem.
Our approach is inspired by that of Choi and Effros in \cite{Choi-Effros}.
The following is a Banach lattice version of  \cite[Lemma 2.4]{Choi-Effros}, and is sort of a dual version of Theorem \ref{thm-M-projection-on-Lr}.

\begin{lemma}\label{lemma-tensor-product-of-L-summand}
Suppose that $X$ and $Y$ are Banach lattices and that $Y_0 \subseteq Y$ is a Banach sublattice such that there exists a positive contractive projection $P$ from $Y$ onto $Y_0$.
If $\iota : Y_0 \to Y$ is the inclusion map, then
$$
Id_X \otimes \iota : X \otimes_{|\pi|} Y_0 \to X \otimes_{|\pi|} Y
$$
is an isometry onto its range.
If additionally $Y_0$ is an order $L$-summand in $Y$ and $P$ is its associated order $L$-projection, then
the range of $Id_X \otimes \iota$ is an order $L$-summand in $X \otimes_{|\pi|} Y$ with associated $L$-projection $Id_X \otimes P$.
\end{lemma}

\begin{proof}
Notice that both
$$
Id_X \otimes P : X \otimes_{|\pi|} Y \to X \otimes_{|\pi|} Y_0 \qquad \text{and} \qquad Id_X \otimes \iota : X \otimes_{|\pi|} Y_0 \to X \otimes_{|\pi|} Y
$$
are positive contractions, because so are $Id_X$, $P$ and $\iota$.
Since $(Id_X \otimes P) (Id_X \otimes \iota) = Id_X \otimes Id_{Y_0}$, it follows that $Id_X \otimes \iota$ is an isometry onto its range, and $X \otimes_{|\pi|} Y_0$ can be regarded as a subspace of $X \otimes_{|\pi|} Y$.

Now assume additionally that $P$ is an order $L$-projection.
Let $u \in X \otimes Y$, and let $x_j \in X_+$, $y_j \in Y_+$ satisfy $\sum x_j \otimes y_j \pm u \in C_p$, where $C_p \subset X \otimes Y$ is the projective cone.
Note that since both $P$ and $Id_Y-P$ are positive, we have
$$
\sum x_j \otimes Py_j \pm (Id_X \otimes P)u \in C_p, \quad \sum x_j \otimes (Id_Y-P)y_j \pm (Id_X \otimes (Id_Y-P))u \in C_p.
$$
Therefore,
\begin{multline*}
\sum \n{x_j} \cdot \n{y_j} = \sum \n{x_j} \cdot \n{Py_j} +  \sum \n{x_j} \cdot \n{(Id_Y-P)y_j} \\
 \ge \n{ (Id_X \otimes P)u }_{|\pi|} + \n{u - (Id_X \otimes P)u}_{|\pi|},
\end{multline*}
and hence $Id_X \otimes P$ is an $L$-projection.
Since $0 \le P \le Id_Y$, it follows that $0 \le Id_X \otimes P \le Id_X \otimes Id_Y$, meaning that $Id_X \otimes P$ is in fact an order $L$-projection.
\end{proof}

The following is a lattice version of  \cite[Lemma 2.5]{Choi-Effros}.
It states that if we restrict our attention to finite-dimensional domains, 
and we start with a map which is almost a lifting on a vector sublattice, we can obtain an actual lifting on the entire domain by doing a small perturbation (and we have control on the size of the perturbation).
A bit of notation: if $K$ is a convex subset of a topological vector space containing $0$, we denote by $\Aff_0(K)$ the Banach subspace of affine functions in $C(K)$ vanishing at $0$.
If $X$ is a Banach space and $K$ is the unit ball of $X^*$ endowed with the weak$^*$ topology, then there is a natural isometry $X \cong \Aff_0(K)$.
 
\begin{lemma}\label{lemma-CE}
Suppose that $J$ is an order $M$-ideal in a Banach lattice $X$.
Let $F \subseteq E$ be finite-dimensional Banach lattices, and let $T : E \to X/J$ be a regular 
map with $\n{T}_r \le 1$.
Given $\eps>0$,
if there exists  
$L : E \to X$ such that $\n{L}_r \le 1$ and $\n{\restr{(q \circ L - T)}{F}}_r \le \eps$,
then there exists $\tilde{L} : E \to X$ such that $T = q \circ \tilde{L}$, $\bign{ \tilde{L} }_r \le 1$ and $\n{\restr{(\tilde{L}-L)}{F}}_r \le 6\eps$.
\end{lemma}

\begin{proof}
Let $\iota : F \to E$ be the inclusion map.
Consider the commutative diagram
\begin{equation}\label{eqn-lemma-CE-diagram-regular}
\xymatrix{
\Lr(F,X/J) & \Lr(E,X/J) \ar[l]_{\circ \iota} \\
\Lr(F,X) \ar[u]^{q \circ } & \Lr(E,X) \ar[l]_{\circ \iota} \ar[u]^{q \circ }\\
}
\end{equation}
Taking adjoints, using the identification from \cite[Lemma 3.4]{Blanco}
\begin{equation}\label{eqn-lemma-CE-diagram-tensors}
\xymatrix{
F \otimes_{|\pi|}(X/J)^* \ar[r]^{\iota \otimes Id_{(X/J)^*}}  \ar[d]^{Id_F \otimes q^*} & E \otimes_{|\pi|} (X/J)^*  \ar[d]^{Id_E \otimes q^*} \\
F\otimes_{|\pi|}X^*\ar[r]^{\iota \otimes Id_{X^*}}  & E \otimes_{|\pi|} X^* \\
}
\end{equation}
Since $q^* : (X/J)^* \to X^*$ is an isometric lattice isomorphism onto a sublattice of $X^*$ which is positively $1$-complemented, it follows from Lemma \ref{lemma-tensor-product-of-L-summand} that both vertical arrows in the diagram \eqref{eqn-lemma-CE-diagram-tensors} are isometric lattice isomorphisms onto their images.
Let $K$ (resp. $\tilde{D}$) be the closed unit ball in $E \otimes_{|\pi|} X^*$ (resp. $F \otimes_{|\pi|} X^*$), and let $D = (\iota \otimes Id_{X^*})\tilde{D}$.
Since $\tilde{D}$ is weak$^*$-compact and $\iota \otimes Id_{X^*}$ is weak$^*$-to-weak$^*$ continuous, 
we see that $D$ is a weak$^*$-closed,  convex, symmetric subset of $K$.

Let $P : X^* \to X^*$ be the $L$-projection onto $J^\perp$, and let
$$
W = (Id_E \otimes P)(E \otimes_{|\pi|} X^*), \qquad \tilde{W} = (Id_F \otimes P)(F \otimes_{|\pi|} X^*)
$$
By Lemma \ref{lemma-tensor-product-of-L-summand}, $W$ is a weak$^*$-closed (order) $L$-summand in $E \otimes_{|\pi|} X^*$ whose corresponding $L$-projection is $Id_E \otimes P$.
Moreover, $W$ (resp. $\tilde{W}$) is the range of $Id_E \otimes q^*$ (resp. $Id_F \otimes q^*$) because the range of $q^*$ is precisely $J^\perp$.

Let $H = K \cap W$; recalling that the vertical arrows in \eqref{eqn-lemma-CE-diagram-tensors} are isometries onto their images, observe that we can identify $H$ with the closed unit ball of $E \otimes_{|\pi|} (X/J)^*$
and similarly we can identify the closed unit ball of $F \otimes_{|\pi|} (X/J)^*$ with $\tilde{D} \cap \tilde{W}$.

We have
\begin{equation}\label{eqn-lemma-CE-commutativity}
(\iota \otimes Id_{X^*})(Id_F \otimes P) = (Id_E \otimes P)(\iota \otimes Id_{X^*}),
\end{equation}
and therefore
$$
(\iota \otimes Id_{X^*})\tilde{W} = (\iota \otimes Id_{X^*})(F \otimes_{|\pi|} X^*) \cap W.
$$
Moreover, since $\iota \otimes Id_{X^*}$ maps $\tilde{D}$ onto $D \subseteq (\iota \otimes Id_{X^*})(F \otimes_{|\pi|} X^*)$, it follows that
\begin{equation}\label{eqn-lemma-CE-intersection}
(\iota \otimes Id_{X^*})(\tilde{D} \cap \tilde{W}) \subseteq D \cap (\iota \otimes Id_{X^*}) \tilde{W} \subseteq D \cap W = D \cap H.
\end{equation}

Consider the commutative diagram of inclusion maps
$$
\xymatrix{
D \cap H \ar[r] \ar[d] &H \ar[d] \\
D \ar[r] &K,
}
$$
and observe that the diagram \eqref{eqn-lemma-CE-diagram-regular} can be identified with the diagram of restriction maps
\begin{equation}\label{eqn-lemma-CE-diagram-affine}
\xymatrix{
\Aff_0(D \cap H) &\Aff_0(H) \ar[l] \\
\Aff_0(D) \ar[u] &\Aff_0(K) \ar[l] \ar[u]
}
\end{equation}
Since $Id_F \otimes P : F\otimes_{|\pi|}X^* \to F\otimes_{|\pi|}X^*$ maps $\tilde{D}$ onto $\tilde{D} \cap \tilde{W}$, it follows from \eqref{eqn-lemma-CE-commutativity} and \eqref{eqn-lemma-CE-intersection} that $Id_E \otimes P$ maps $D$ onto $D \cap H$,
and therefore the conditions of \cite[Lemma 2.1]{Choi-Effros} are satisfied.

Let us now reinterpret the hypotheses in terms of the diagram \eqref{eqn-lemma-CE-diagram-affine}.
We are given $a \in \Aff_0(H)$ with $\n{a} \le 1$ (corresponding to $T$) and $b \in \Aff_0(K)$ with $\n{b} \le 1$ (corresponding to $L$) such that $\n{\restr{(a-b)}{D \cap H}} \le \eps$.
Therefore, by \cite[Lemma 2.1]{Choi-Effros}, there exists $\tilde{a} \in \Aff_0(K)$ (corresponding to $\tilde{L} \in \Lr(E,X)$) such that $\n{\tilde{a}} \le 1$, $\restr{\tilde{a}}{H} = a$ (corresponding to $q \circ L = T$) and $\n{\restr{(a-b)}{D}} \le 6\eps$ (corresponding to $\n{ \restr{(\tilde{L}-L)}{F} }_r \le 6\eps$).
\end{proof}

It is well-known that in general, a finite-dimensional subspace of a Banach lattice is not necessarily contained in a finite-dimensional vector sublattice.
However, under suitable completeness assumptions this can almost be achieved: any finite-dimensional subspace can be placed inside a finite-dimensional vector sublattice by ``moving it a little bit'' using a linear operator (see, e.g. \cite[Prop. 2.1]{Blanco}).
The following is a variation on a lemma of this type due to Lissitsin and Oja \cite[Lemma 5.5]{Lissitsin-Oja}, where the linear operator has extra structure.

\begin{lemma}\label{lemma-LO}
Let $F$ be a finite-dimensional subspace of a Dedekind complete Banach lattice $X$ and let $\eps>0$. Then there exist a sublattice $Z$ of $X$ containing $F$, a finite-dimensional sublattice $G$ of $Z$, and a lattice-homomorphic projection $P$ from $Z$ onto $G$ such that $\n{ \restr{(P-Id)}{F} } \le \eps$.
If $F$ contains a vector sublattice $A$ of $X$, we can additionally arrange to have $\n{ \restr{(P-Id)}{A} }_r \le \eps$.
\end{lemma}

\begin{proof}
The statement in \cite{Lissitsin-Oja} asks for $X$ to be order continuous, but the proof only requires Dedekind completeness as can be seen in the proof of the related result \cite[Lemma 2.4]{Blanco}.
Additionally \cite[Lemma 5.5]{Lissitsin-Oja} only has $P$ being a positive projection, but it is clear from their proof that $P$ is also a lattice homomorphism.
The only other thing missing in \cite{Lissitsin-Oja} is the small regular norm,
which follows from \cite[Lemma 2.4]{Blanco}.
\end{proof}

The next preparatory lemma will allow us to define a regular map on a Banach lattice in a step-by-step fashion, by defining it on larger and larger vector sublattices.
As with Proposition \ref{prop-PLRL-a-la-Dean}, this might be folklore but we have been unable to locate a reference.

\begin{lemma}\label{lemma-regular-defined-by-pieces}
Let $Y$, $Z$ be Banach lattices with $Z$ Dedekind complete.
Suppose $(Y_n)$ is an increasing sequence of vector sublattices of $Y$ such that $\bigcup_{n=1}^\infty (Y_n)_+$ is dense in $Y_+$, and let $T : \bigcup_{n=1}^\infty Y_n \to Z$ be a linear operator such that for each $n\in\N$ we have $\n{ \restr{T}{Y_n} : Y_n \to Z }_r \le 1$.
Then $T$ extends to a bounded linear operator $\tilde{T} : Y \to Z$ with $\bign{\tilde{T}}_r \le 1$. 
\end{lemma}

\begin{proof}
For simplicity, let us define $T_n = \restr{T}{Y_n} : Y_n \to Z$.
Notice that $T$ uniquely extends to a continuous $\tilde{T} : Y \to Z$, so we only need to check that $\tilde{T}$ has regular norm at most one.
Since $Z$ is Dedekind complete, the regularity of a $Z$-valued map is equivalent to it being order bounded or having a modulus, see \cite[Prop. IV.1.2]{Schaefer}.

Let $x \in Y_+$, and let $\eps>0$ be given.
Set $x_0=x$. Assuming $x_k \in Y_+$ has been chosen, find  $n_k \in \N$ and $x_k' \in (Y_{n_k})_+$ such that $\n{x_k-x'_k} < \eps/2^{k+1}$, and set $x_{k+1} = |x_k-x'_k|$.
Observe that both of the series
$$
\sum_{k=1}^\infty x_k, \quad \sum_{k=1}^\infty x'_k
$$
converge absolutely, since for any $k$ we have $\n{x_{k+1}} < \eps/2^{k+1}$ and
$\n{x'_k} \le \n{x_k} + \n{x_k-x'_k} \le \n{x_k} + \eps/2^{k+1}$.

Now let $y \in Y$ such that $0 \le y \le x$.
Set $y_0=y$, and
assume $0 \le y_k \le x_k$ has been chosen.
Observe that $0 \le y_k \le x_k \le x'_k + x_{k+1}$, since $x_{k+1} = |x_k-x'_k|$.
Therefore, by the Riesz decomposition property \cite[Thm. 1.15]{Aliprantis-Burkinshaw}, we can write $y_k = y'_k + y_{k+1}$ with $0 \le y'_k \le x'_k$ and $0 \le y_{k+1} \le x_{k+1}$. 

Observe that the series
$
\sum_{k=1}^\infty y_k'
$
also converges absolutely, since for any $k$ we have
$y'_{k+2} \le y_{k+1} \le x_{k+1}$, hence $\n{y'_{k+2}} \le \n{x_{k+1}} < \eps/2^{k+2}$.
Moreover, note that $y = \sum_{k=1}^\infty y'_k$.
Therefore,
$$
|\tilde{T}y| = \Big| \sum_{k=1}^\infty Ty'_k \Big| \le \sum_{k=1}^\infty |Ty'_k| \le \sum_{k=1}^\infty |T_{n_k}| y'_k \le \sum_{k=1}^\infty |T_{n_k}| x'_k
$$
where the last series converges absolutely because all the operators $|T_{n_k}|$ are contractions.

This shows that $\tilde{T}$ is order bounded, and since $Z$ is Dedekind complete, $\tilde{T}$ has a modulus and is therefore regular.
Moreover, the above calculations show that for $x \in Y_+$ we have
$$
\bign{ |\tilde{T}| x } \le  \n{x} + \eps
$$
and therefore
$$
\bign{\tilde{T}}_r = \bign{ \,|\tilde{T}|\, } \le 1.
$$
\end{proof}

\section{Ando-Choi-Effros liftings for regular maps under the BPAP}\label{sec-main}

We are now ready to prove our first Ando-Choi-Effros lifting theorem for regular maps between Banach lattices.
The argument is similar to that of \cite[Thm. 2.6]{Choi-Effros} (which in turn was inspired by \cite[Prop. 5]{Andersen}),
but it is significantly more involved due to the aforementioned fact that 
a finite-dimensional subspace of a Banach lattice is not necessarily contained in a finite-dimensional vector sublattice.
This is also why we are using  \cite[Thm. 2.6]{Choi-Effros} as a model, instead of a cleaner proof such as that of \cite[Thm. II.2.1]{Harmand-Werner-Werner}: the ``wiggle'' factor coming from Lemma \ref{lemma-LO} appears to render those cleaner arguments inaccessible.

\begin{theorem}\label{thm-Ando-Choi-Effros-regular maps}
	Suppose that $J$ is an order $M$-ideal in the Dedekind complete Banach lattice $X$, and let $q : X \to X/J$ be the canonical quotient map.
	Let $Y$ be a separable and Dedekind complete Banach lattice, and let $T : Y \to X/J$ be a regular map with $\n{T}_r = 1$.
	If $Y$ has the $\lambda$-BPAP,
	then there exists $L : Y \to X$ such that $q \circ L = T$
	and $\n{L}_r \le \lambda$.
\end{theorem}

\begin{proof} 
To simplify the writing of the various estimates we will assume $\lambda =1$, but the same argument works for any $\lambda \ge 1$.
Fix a dense sequence $\{y_n\}_{n=1}^\infty$ in the unit sphere of $Y$.
We inductively define a sequence $(G_k)$ of finite-dimensional vector sublattices of $Y$, positive maps $S_k : Y \to G_k$ and lattice isomorphisms onto their images $j_k : G_{k} \to G_{k+1}$ as follows.
Let $G_0 = \{0\}$ and $S_0 = 0$.
Having defined $G_k$, $S_k$ and $j_{k-1}$ for an integer $k \ge 0$ (with the convention $j_{-1} = 0$),
use the 1-BPAP to find a finite-rank positive map $R_{k+1} : Y \to Y$ such that $\n{R_{k+1}} \le 1$ and
$$
\n{ \restr{(R_{k+1} - Id)}{G_k} }_r  < 2^{-(k+1)}.
$$
It should be noted that the BPAP only gives small operator norm, but we can obtain small regular norm by \cite[Lemma 2.4]{Blanco}
as we did in the proof of Lemma \ref{lemma-LO}. 
Consider the finite-dimensional subspace $F_{k+1} = G_k + R_{k+1}(Y) + \R y_{k+1}$, and apply Lemma \ref{lemma-LO} to find a sublattice $Z_{k+1}$ of $Y$ containing $F_{k+1}$, and  a lattice homomorphic projection $P_{k+1} : Z_{k+1}\to Z_{k+1}$ onto a finite-dimensional vector sublattice $G_{k+1} \subset Z_{k+1}$ such that
 $$
 \n{ \restr{ (P_{k+1}-Id)}{F_{k+1}} } < 2^{-(k+1)} \quad \text{and} \quad \n{ \restr{ (P_{k+1}-Id)}{G_k} }_r < 2^{-(k+1)}.
 $$ 
Define $j_{k} : G_{k} \to G_{k+1}$ as the restriction of $P_{k+1}$ to $G_k$, and observe that 
\begin{equation}\label{eqn-ACE-jk}
\n{ j_k -\restr{Id}{G_{k}}}_r <  2^{-(k+1)}, \text{ so } \n{j_{k}} = \n{j_{k}}_r < 1 + 2^{-(k+1)}.
\end{equation}
Define $S_{k+1} : Y \to G_{k+1}$ as $S_{k+1} = P_{k+1} \circ R_{k+1}$, and observe that $S_{k+1}$ has finite rank, is positive,
$\n{ S_{k+1} } < 1 +  2^{-(k+1)}$ and
\begin{multline}\label{eqn-ACE-Sk}
\n{ \restr{(S_{k+1} - Id)}{G_k} }_r  = \n{ \restr{(P_{k+1} \circ R_{k+1} - Id)}{G_k} }_r \\
\le \n{ \restr{(P_{k+1} \circ (R_{k+1} - Id) )}{G_k} }_r +  \n{ \restr{(P_{k+1} - Id)}{G_k} }_r \\
< \n{\restr{P_{k+1}}{F_{k+1}}} \cdot \n{ \restr{(R_{k+1} - Id)}{G_k} }_r  + \frac{1}{2^{k+1}} < 2 \cdot \frac{1}{2^{k+1}} + \frac{1}{2^{k+1}} = \frac{3}{2^{k+1}}.
\end{multline}

Next, we construct inductively a sequence of maps $L_k : G_k \to X$ such that $\n{L_k}_r \le 1$ and $q \circ L_k = \restr{T}{G_k}$.
Let $L_0 = 0$, and assume we have defined such a map $L_k$ for a particular integer $k \ge 0$.
With the aim of applying Lemma \ref{lemma-CE}, now consider (with the convention that anything with subindex $-1$ or $0$ is taken to be zero)
the subspace $j_kj_{k-1}(G_{k-1})$ of $G_{k+1}$
and the maps 
$$
\restr{T}{G_{k+1}} : G_{k+1} \to X/J \quad  \text{and} \quad \frac{1}{1 +  2^{-(k+1)}} L_k \circ \restr{S_k}{G_{k+1}} : G_{k+1} \to X,
$$
both of which have regular norm at most one.
Now,
\begin{multline}\label{eqn-ACE-regular-diff-1}
\n{ \restr{ (q \circ  \tfrac{1}{1 +  2^{-(k+1)}} L_k \circ S_k - T ) } {j_kj_{k-1}(G_{k-1})}  }_r \\
\le \n{ \restr{ (q \circ L_k \circ S_k - T ) } {j_kj_{k-1}(G_{k-1})}  }_r  + \n{ \big( 1-\tfrac{1}{1 +  2^{-(k+1)}} \big) q \circ L_k \circ S_k }_r \\
\le  \n{ \restr{ (q \circ L_k \circ S_k - T ) } {j_kj_{k-1}(G_{k-1})}  }_r + \tfrac{2^{-(k+1)}}{1 +  2^{-(k+1)}} \cdot 1 \cdot 1 \cdot (1 +  2^{-(k+1)}) \\
\le  \n{ \restr{ (q \circ L_k \circ S_k - T ) } {j_kj_{k-1}(G_{k-1})}  }_r + 2^{-(k+1)}.
\end{multline}
On the other hand,
\begin{multline}\label{eqn-ACE-regular-diff-2}
\n{ \restr{ (q \circ L_k \circ S_k - T ) } {j_kj_{k-1}(G_{k-1})}  }_r = \n{ \restr{ (T \circ (S_k - Id) ) } {j_kj_{k-1}(G_{k-1})}  }_r \le \\
\n{T}_r \n{ \restr{ (S_k - Id) } {j_kj_{k-1}(G_{k-1})}  }_r  
\le 4  \n{ \restr{(S_kj_kj_{k-1}- j_kj_{k-1})}{G_{k-1}} }_r,
\end{multline}
where we have used that $\n{j_m^{-1}}= \n{j_m^{-1}}_r\le 2$ for $m \ge 1$.
Since
\begin{multline*}
S_kj_kj_{k-1}-j_kj_{k-1} = S_k(j_k-Id)j_{k-1} + S_k(j_{k-1}-Id)+(S_k-Id)+(Id-j_{k-1})+(Id-j_k)j_{k-1},
\end{multline*}
using \eqref{eqn-ACE-jk}, \eqref{eqn-ACE-Sk} and $\n{S_k} \le 2$ we conclude
$$
\n{ \restr{(S_kj_kj_{k-1}- j_kj_{k-1})}{G_{k-1}} }_r < \frac{12}{2^k}
$$
which, together with \eqref{eqn-ACE-regular-diff-1} and \eqref{eqn-ACE-regular-diff-2} implies
$$
\n{ \restr{ (q \circ  \tfrac{1}{1 +  2^{-(k+1)}} L_k \circ S_k - T ) } {j_kj_{k-1}(G_{k-1})}  }_r < \frac{49}{2^k}.
$$
Therefore, by Lemma \ref{lemma-CE} there exists $L_{k+1} : G_{k+1} \to X$ such that $\n{L_{k+1}}_r \le 1$, ${q \circ L_{k+1} = \restr{T}{G_{k+1}}}$ and
$$
\n{ \restr{(L_{k+1} - \tfrac{1}{1 +  2^{-(k+1)}} L_k \circ S_k) }{j_kj_{k-1}(G_{k-1})} }_r <  \frac{294}{2^k},
$$
from where it follows that 
\begin{multline}\label{eqn-ACE-regular-difference-of-Ls}
\n{ \restr{(L_{k+1} - L_k \circ S_k) }{j_kj_{k-1}(G_{k-1})} }_r  \\
\le \n{ \restr{(L_{k+1} - \tfrac{1}{1 +  2^{-(k+1)}} L_k \circ S_k) }{j_kj_{k-1}(G_{k-1})} }_r + \n{ \restr{ (1-\tfrac{1}{1 +  2^{-(k+1)}})  L_k \circ S_k) }{j_kj_{k-1}(G_{k-1})} }_r < \frac{295}{2^k}.
\end{multline}

Notice that since $\sum_{k=1}^\infty 2^{-k}$ converges, so does the infinite product
$
\prod_{k = 1}^\infty (1+2^{-k})
$.
Let $C_n = \prod_{k = n}^\infty (1+2^{-k})$, and observe that $C_n$ converges to 1.

Fix a number $n_0 \in \N$.
Consider the sequence of operators $\big(  L_{k+1} j_{k} \cdots j_{n_0+1} j_{n_0}  \big)_{k > n_0+2}$ in $\Lr(G_{n_0},X)$.

From \eqref{eqn-ACE-jk} and \eqref{eqn-ACE-Sk}, and observing that for $x \in G_k$
$$
|S_{k+1}j_k x - j_k x| \le |S_{k+1}j_kx - S_{k+1}x| + |S_{k+1}x - x| + |x - j_kx|,
$$
 it follows that
$$
\n{ \restr{(S_{k+1} - Id)}{j_k(G_k)} }_r < 5 \cdot 2^{-(k+1)},
$$
and by an analogous argument we get
\begin{equation}\label{eqn-ACE-Sk-adjusted}
\n{ \restr{(S_{k+1} - Id)}{j_{k+1}j_k(G_k)} }_r < 6 \cdot 2^{-(k+1)}.
\end{equation}
Now,
\begin{multline*}
\n{ L_{k+2} j_{k+1}j_k \cdots j_{n_0+1} j_{n_0}  - L_{k+1} j_{k} \cdots j_{n_0+1} j_{n_0}  }_r \\
\le  \n{j_{n_0}} \cdot \n{j_{n_0+1}} \cdots \n{j_{k}} \cdot \n{ \restr{(L_{k+2}j_{k+1} - L_{k+1})}{ j_kj_{k-1}(G_{k-1}) } }_r,
\end{multline*}
and using \eqref{eqn-ACE-regular-difference-of-Ls}, \eqref{eqn-ACE-Sk-adjusted} and \eqref{eqn-ACE-jk},
\begin{multline*} 
\n{ \restr{(L_{k+2}j_{k+1} - L_{k+1})}{ j_kj_{k-1}(G_{k-1}) } }_r \\
 \le \n{ \restr{(L_{k+2}j_{k+1} - L_{k+1} S_{k+1} j_{k+1} )}{ j_kj_{k-1}(G_{k-1}) } }_r
+ \n{ \restr{(L_{k+1} S_{k+1} j_{k+1} - L_{k+1} )}{ j_kj_{k-1}(G_{k-1}) } }_r \\
\le \n{j_{k+1}} \cdot  \n{ \restr{(L_{k+2} - L_{k+1} S_{k+1} )}{ j_{k+1}j_k(G_{k}) } }_r
+ \n{L_{k+1}}_r  \n{ \restr{( S_{k+1} j_{k+1} - Id )}{ j_kj_{k-1}(G_{k-1}) } }_r \\
< 2 \cdot 295\cdot 2^{-(k+1)} + \n{ \restr{( S_{k+1} j_{k+1} - j_{k+1} )}{ j_kj_{k-1}(G_{k-1}) } }_r + \n{ \restr{( j_{k+1} - Id )}{ j_kj_{k-1}(G_{k-1}) } }_r \\
< 590\cdot 2^{-(k+1)} + 2 \cdot 6 \cdot 2^{-(k+1)} + 2^{-(k+2)} < 603 \cdot 2^{-(k+1)}.
\end{multline*}
It therefore follows that
\begin{equation}\label{eqn-ACE-Cauchy}
\n{ L_{k+2} j_{k+1}j_k \cdots j_{n_0+1} j_{n_0}  - L_{k+1} j_{k} \cdots j_{n_0+1} j_{n_0}  }_r 
< C_{n_0} \cdot 603 \cdot 2^{-(k+1)},
\end{equation}
hence the sequence $\big(  L_{k+1} j_{k} \cdots j_{n_0+1} j_{n_0}  \big)_{k > n_0+2}$ is Cauchy in $\Lr(G_{n_0},X)$ and converges to a limit $\tilde{L}_{n_0} \in \Lr(G_{n_0},X)$ satisfying $\n{\tilde{L}_{n_0}}_r \le C_{n_0}$.
Moreover, the operators $\tilde{L}_{k}$ are ``compatible'' with the $j_k$ in the sense that for every $k$ we have
$$
\tilde{L}_k = \tilde{L}_{k+1} j_k.
$$
Let
$$
c(G_k) = \big\{  (y_k)  \st y_k \in G_k \text{ for each } k\in\N,  \text{ and } (y_k) \text{ converges in $Y$}  \big\},
$$
endowed with the supremum norm and the coordinatewise order; 
observe that this is a Banach lattice. For each $n \in \N$, define
$$
A_n = \big\{ (y_1, y_2, \dotsc, y_{n-1}, y_n, j_ny_n, j_{n+1}j_ny_n, j_{n+2}j_{n+1}j_ny_n, \dotsc) \; : \; y_j \in G_j \text{ for $1 \le j \le n$}  \big\}
$$
and observe that  $(A_n)_{n=1}^\infty$
is an increasing sequence of Banach sublattices of $c(G_k)$.

Now define an operator $\hat{L} : \bigcup_{n=1}^\infty A_n \to c(X)$ by $\hat{L}(y_k)_k = ( \tilde{L}_ky_k )_k$.
If $(y_k)_k \in A_n$, note that
$$
\hat{L}(y_k) =  (\tilde{L}_1y_1, \tilde{L}_2y_2, \dotsc, \tilde{L}_{n-1}y_{n-1}, \tilde{L}_ny_n, \tilde{L}_ny_n, \tilde{L}_ny_n, \dotsc)
$$
because of the compatibility conditions above, and
therefore 
$$
\n{  \restr{ \hat{L} }{ A_n } }_r \le 1.
$$
Since $\bigcup_{n=1}^\infty (A_n)_+$ is dense in $\big(c(G_k)\big)_+$, by Lemma \ref{lemma-regular-defined-by-pieces} we have that 
$\hat{L}$ extends to a regular operator from $c(G_k)$ to $c(X)$, which we will again denote by $\hat{L}$,
having regular norm at most one.

Let $Q : c(X) \to X$ be given by $Q(x_k)_k = \lim_{k\to\infty} x_k$, which is a positive contraction.
Define $S : Y \to c(G_k)$ by $Sy = (S_ky)_k$, and observe that it is a positive operator with norm at most $\sup_k \n{S_k}$.
If we define $L = Q \circ \hat{L} \circ S$, it is clear that $L$ is a regular map such that $q \circ L = T$ but we  can only guarantee $\n{L}_r \le \sup_k \n{S_k}$.
However, note that if we fix $k_0 \in\N$ the operator $L$ does not change when we replace $S$ by  the map $Y \to c(G_k)_{k \ge k_0}$ given by $y \mapsto (S_ky)_{k\ge k_0}$.
Therefore for each $k_0\in\N$ we have $\n{L}_r \le \sup_{k\ge k_0} \n{S_k}$, which implies $\n{L}_r \le 1$.
\end{proof}

\begin{remark}
In Theorem \ref{thm-Ando-Choi-Effros-regular maps}, since we are working with a Dedekind complete Banach lattice $Y$, requiring the BPAP is equivalent to requiring the bounded lattice approximation property  (see \cite[Cor. 4.3]{Blanco}).
\end{remark}

\section{Ando-Choi-Effros liftings for regular maps under Cartwright's property (C)}\label{sec-main-2}

As already mentioned in the introduction there is a second version of the Ando-Choi-Effros theorem, where the domain of the map to be extended is assumed to be an $L_1$-predual instead of having the BAP.
Going through the proof (see, e.g. \cite[Thm. II.2.1]{Harmand-Werner-Werner}), it is easy to see that the key property of $L_1$-preduals used in the argument is the fact that their biduals are injective Banach spaces.
In the context of lattices, instead of $L_1$-preduals the natural choice would be to consider lattices with Cartwright's property (C).
Recall that a Banach lattice $X$ has property (C) if whenever $x_1, x_2, y \in X_+$ and real numbers $r_1, r_2$ satisfy
$$
\n{x_i} \le r_i, \quad \n{x_1 + x_2 + y} \le r_1 +r_2,
$$
then there exist $y_1, y_2 \in X_+$ such that $y = y_1 + y_2$ and $\n{x_i + y_i} \le r_i$.
Cartwright proved that a Banach lattice $X$ has property (C) if and only if $X^{**}$ is an injective Banach lattice \cite{Cartwright}.
We prove below a version of Theorem \ref{thm-Ando-Choi-Effros-regular maps} where the Banach lattice $Y$ is assumed to have property (C) instead of the $\lambda$-BPAP.
This time our proof is inspired by \cite{Ando-nonempty} rather than \cite{Choi-Effros}, and the presentation borrows heavily from \cite{Harmand-Werner-Werner}.

The following preparatory lemma is an adaptation of \cite[Lemma. II.2.4]{Harmand-Werner-Werner},
and deals with the fundamental step of extending a lifting defined on a finite-dimensional lattice to a larger finite-dimensional lattice. 

\begin{lemma}\label{lemma-pre-Ando-Choi-Effros}
	Suppose that $J$ is an order $M$-ideal in the Banach lattice $X$, and let $q : X \to X/J$ be the canonical quotient map.
	Let $F \subseteq E$ be finite-dimensional Banach lattices,
	and let $T : E \to X/J$ be a regular map with $\n{T}_r = 1$.
	Assume that $J$ satisfies property (C).
	Then, given $\eps>0$ and a map $L : F \to X$ with $\n{L}_r \le 1$ such that $ q \circ L = \restr{T}{F}$, there exists a map
$\tilde{L} : E \to X$ such that $\| \tilde{L} \|_r\le 1$, $q \circ \tilde{L} = T$ and $\bign{ \restr{\tilde{L}}{F} - L}_r \le \eps$.
\end{lemma}

\begin{proof}
We start by defining
\begin{align*}
W &= \big\{ S \in \Lr(E,X) \mid \ran(S) \subseteq J  \} \equiv \Lr(E,J)	\\
V &= \big\{ S \in W \mid  \ker(S) \supseteq F \}.
\end{align*}
	
	Using Proposition \ref{prop-PLRL-a-la-Dean}, it follows that
\begin{align*}
W^{\perp\perp} &= \big\{ S \in \Lr(E,X^{**}) \mid \ran(S) \subseteq J^{\perp\perp}  \} \equiv \Lr(E,J^{\perp\perp})	\\
V^{\perp\perp} &= \big\{ S \in W^{\perp\perp} \mid  \ker(S) \supseteq F \}
\end{align*}	
					
	Let us now observe that $W$ is an $M$-ideal in $\Lr(E,X)$.
	By Theorem \ref{thm-M-projection-on-Lr},
	if $P : X^{**} \to J^{\perp\perp}$ is the order $M$-projection associated with $J$, then $\tilde{P} : S \mapsto P \circ S$ is an $M$-projection
	on $\Lr(E,X^{**})$.
	The range of $\tilde{P}$ is obviously contained in $W^{\perp\perp}$, and it is easy to see that the range is in fact all of $W^{\perp\perp}$.
	Since $W^{\perp\perp}$ is weak$^*$-closed, it follows from \cite[Cor. II.3.6]{Harmand-Werner-Werner} that $\tilde{P}$ is the adjoint of an $L$-projection and therefore $W$ is an $M$-ideal in $\Lr(E,X)$.

	Now let $L_1 \in \Lr(E,X)$ be any extension of $L$ such that $q \circ L_1 = T$
	this exists because $E$ is finite-dimensional, and can be achieved by a completing-the-basis argument.
	Let $B$ denote the unit ball of $\Lr(E,X)$. We would like to prove that
	\begin{equation}\label{L-in-closure}
	L_1 \in \overline{ B+V }.
	\end{equation}
	In order to achieve it, we will consider $L_1$ as an element of $\Lr(E,X^{**}) \equiv \Lr(E,X)^{**}$ and we will show that
	\begin{equation}\label{L-in-weak-star-closure}
	L_1 \in \overline{ B+V }^{w^*}.
	\end{equation}

Recall that our assumption on $J$  implies that $J^{\perp\perp}$ is an injective Banach lattice.
Since every dual Banach lattice is Dedekind complete,
	by \cite[Thm. 2.2]{Arendt} there exists an extension $\Lambda : E \to J^{\perp\perp}$ of $PL : F \to J^{\perp\perp}$ with $\n{\Lambda}_r \le 1$ .	Let us decompose $L_1$ as
	$$
	L_1 = \big( (Id_{X^{**}} - P) L_1 +  \Lambda \big) + (PL_1 - \Lambda),
	$$ 

	First note that $PL_1 - \Lambda \in V^{\perp\perp}$.
	Since $\ran(Id_{X^{**}}-P) \equiv (X/J)^{**}$, and looking at the diagram
	$$
	\xymatrix{
	E \ar[dr]_{L_1}\ar[r]^T &X/J \ar[r] & (X/J)^{**}\\
			&X \ar[u]^{q}\ar[r] &X^{**} \ar[u]_{Id_{X^{**}}-P}
	}
	$$
	it follows that $\n{(Id_{X^{**}} - P) L_1}_r = \n{T}_r = 1$.
	Since $\tilde{P}$ is an $M$-projection on $\Lr(E,X^{**})$, and $P\Lambda= \Lambda$, it follows that
	$$
	\n{(Id_{X^{**}} - P) L_1 + \Lambda}_r = \max \big\{  \n{(Id_{X^{**}} - P) L_1}_r,  \n{ \Lambda }_r \big\} = 1.
	$$
	Note also that $\Lambda$ and $(Id_{X^{**}} - P) L_1$ both belong to $\Lr(E,X^{**})$.
	Therefore,
	$$
	L_1 \in B_{\Lr(E,X^{**})} + V^{\perp\perp} = \overline{B}^{w^*} + \overline{V}^{w^*} = \overline{B+V}^{w^*}.
	$$
	giving \eqref{L-in-weak-star-closure}.
		
	Now, from \eqref{L-in-closure} there exist $R \in B$ and $S \in V$ such that $\n{ L_1 - (R + S) }_r \le \varepsilon/2$.
	Define $L_2 = L_1 - S \in \Lr(E,X)$. Note that $L_2$ is a lifting for $T$, since $S \in V \subset W$,
	but it is not guaranteed to have regular norm at most one: we only have $\n{L_2}_r \le (1+\varepsilon/2)$.
	We would like to perturb $L_2$ slightly to obtain a map that is still a lifting but whose regular norm is in fact at most 1.
	Now,
	\begin{align*}
	L_2 &\in (L_1 + V) \cap (1+\varepsilon/2)B \subset (\overline{B+V}) \cap (1+\varepsilon/2)B \\
	&\subset (\overline{B+W}) \cap (1+\varepsilon/2)B \subset B + \varepsilon(B \cap W)
	\end{align*}
	where we have used \eqref{L-in-closure} in the last step of the first line, and \cite[Lemma II.2.5]{Harmand-Werner-Werner} in the last step of the second line.
	Thus there exists $\tilde{L} \in \Lr(E,X)$ with $\| \tilde{L} \|_r \le 1$,  $\| \tilde{L} - L_2\|_r \le \varepsilon$ and $\tilde{L} - L_2 \in W$.
	It follows that $\tilde{L}$ satisfies the desired conditions.
\end{proof}

We are now ready to prove our second version of the Ando-Choi-Effros lifting theorem for regular maps.

\begin{theorem}\label{thm-Ando-Choi-Effros-regular maps-injective}
	Suppose that $J$ is an order $M$-ideal in the Banach lattice $X$,  let $q : X \to X/J$ be the canonical quotient map.
	Let $Y$ be a separable Dedekind complete Banach lattice
			and let $T : Y \to X/J$ be a regular map with $\n{T}_r = 1$.
	If $J$ satisfies property (C),
	then there exists $L : Y \to X$ such that $q \circ L = T$
	and $\n{L}_r \le 1$.
\end{theorem}

\begin{proof}
Once again fix a dense sequence $\{y_n\}_{n=1}^\infty$ in the unit sphere of $Y$.
We inductively define a sequence $(G_k)$ of finite-dimensional vector sublattices of $Y$ and lattice isomorphisms onto their images $j_k : G_{k} \to G_{k+1}$ as follows.
Let $G_{0} = \{ 0\}$.
Having defined $G_k$ and $j_{k-1}$ for an integer $k \ge 0$ (with the convention $j_{-1} = 0$),
consider the finite-dimensional subspace $F_{k+1} = G_k + \R y_{k+1}$, and apply Lemma \ref{lemma-LO} to find a sublattice $Z_{k+1}$ of $Y$ containing $F_{k+1}$, and  a lattice homomorphic projection $P_{k+1} : Z_{k+1}\to Z_{k+1}$ onto a finite-dimensional vector sublattice $G_{k+1} \subset Z_{k+1}$ such that
 $$
 \n{ \restr{ (P_{k+1}-Id)}{F_{k+1}} } < 2^{-(k+1)} \quad \text{and} \quad \n{ \restr{ (P_{k+1}-Id)}{G_k} }_r < 2^{-(k+1)}.
 $$ 
Define $j_{k} : G_{k} \to G_{k+1}$ as the restriction of $P_{k+1}$ to $G_k$, and observe that \eqref{eqn-ACE-jk} holds again.
Next, we construct inductively a sequence of maps $L_k : G_k \to X$ such that $\n{L_k}_r \le 1$ and $q \circ L_k = \restr{T}{G_k}$.
Let $L_0 = 0$, and assume we have defined such a map $L_k$ for a particular integer $k \ge 0$.
Observe that
\begin{multline}
\n{ \restr{  (q \circ L_k \circ j_k^{-1} - T)}{ j_k(G_k) } }_r =  \n{ \restr{  (T \circ j_k^{-1} - T)}{ j_k(G_k) } }_r \le \n{ T \circ \restr{(Id - j_k)}{G_k} }_r < 2^{-(k+1)}
\end{multline}
and therefore, using Lemma \ref{lemma-CE}, there exists $\mathbb{L}_{k+1} : j_k(G_k) \to X$ such that $\n{ \mathbb{L}_{k+1} }_r \le $,  $q \circ \mathbb{L}_{k+1} = \restr{T}{j_k(G_k)}$ and 
$\n{ \mathbb{L}_{k+1} - L_k \circ j_k^{-1}  }_r < 6 \cdot 2^{-(k+1)}$.
Now, by Lemma \ref{lemma-pre-Ando-Choi-Effros}, there exists $L_{k+1} : G_{k+1} \to X$ such that $\n{L_{k+1}}_r \le 1$, $q \circ L_{k+1} = \restr{T}{G_{k+1}}$ and $\n{ \restr{L_{k+1}}{j_k(G_k)} - \mathbb{L}_{k+1}  }_r < 2^{-(k+1)}$, from where it follows (using $\n{j_k} \le 2$)
\begin{multline}\label{eqn-ACE-injective-almost-compatible}
\n{ L_{k+1}\circ j_k - L_k }_r \le \n{ L_{k+1}\circ j_k - \mathbb{L}_{k+1} \circ j_k}_r+ \n{  \mathbb{L}_{k+1} \circ j_k -  L_k }_r \\
 \le \n{j_k}\n{ \restr{L_{k+1}}{j_k(G_k)} - \mathbb{L}_{k+1}  }_r + \n{j_k} \n{\mathbb{L}_{k+1} - L_k \circ j_k^{-1}} < 14 \cdot 2^{-(k+1)}.
\end{multline}
Fix a number $n_0 \in \N$, and
consider the sequence of operators $\big(  L_{k+1} j_{k} \cdots j_{n_0+1} j_{n_0}  \big)_{k > n_0}$ in $\Lr(G_{n_0},X)$.
It is easy to see from \eqref{eqn-ACE-injective-almost-compatible} that the sequence is Cauchy, and therefore it converges to a limit 
$\tilde{L}_{n_0} \in \Lr(G_{n_0},X)$. 
Moreover, the operators $\tilde{L}_{k}$ are ``compatible'' with the $j_k$ in the sense that for every $k$ we have
$
\tilde{L}_k = \tilde{L}_{k+1} j_k.
$
The rest of the proof continues in exactly the same way as in the proof of Theorem \ref{thm-Ando-Choi-Effros-regular maps}.
\end{proof}

\begin{remark}
In Theorems \ref{thm-Ando-Choi-Effros-regular maps} and \ref{thm-Ando-Choi-Effros-regular maps-injective},
it would be desirable to have the lifting $L$ be a positive operator when the initial map $T$ is a positive operator.
It is possible that the arguments above already prove such results, but we have been unable to verify it.
In the case of Theorem \ref{thm-Ando-Choi-Effros-regular maps} the key step would be to adapt Lemma \ref{lemma-CE}, where if $T$ and $L$ are positive then $\tilde{L}$ should be chosen positive as well.
To reuse the current proof for Lemma \ref{lemma-CE}, we would need a version of \cite[Lemma 2.1]{Choi-Effros} that takes positivity into account. This would require to prove positivity-preserving versions of the first four results in  \cite[Part I, Sec. 5]{Alfsen-Effros}, 
and that does not appear to be straightforward.
For Theorem \ref{thm-Ando-Choi-Effros-regular maps-injective}, the key would be to obtain a version of Lemma \ref{lemma-pre-Ando-Choi-Effros} where once again $\widetilde{L}$ can be chosen positive when $T$ and $L$ are.
Most of our proof does deal well with positivity, but we have not been able to obtain an appropriate accompanying version of  \cite[Lemma II.2.5]{Harmand-Werner-Werner}.
\end{remark}

\section*{Acknowledgments}
The author thanks Profs. W.B. Johnson, V. Troitsky, A. Blanco and P. Tradacete for useful discussions and for pointing out several references.

\def\cprime{$'$}
\providecommand{\bysame}{\leavevmode\hbox to3em{\hrulefill}\thinspace}
\providecommand{\MR}{\relax\ifhmode\unskip\space\fi MR }
\providecommand{\MRhref}[2]{%
  \href{http://www.ams.org/mathscinet-getitem?mr=#1}{#2}
}
\providecommand{\href}[2]{#2}


\begin{thebibliography}{HWW93}

\bibitem[AB06]{Aliprantis-Burkinshaw}
Charalambos~D. Aliprantis and Owen Burkinshaw, \emph{Positive operators},
  Springer, Dordrecht, 2006, Reprint of the 1985 original. \MR{2262133}

\bibitem[AE72]{Alfsen-Effros}
Erik~M. Alfsen and Edward~G. Effros, \emph{Structure in real {B}anach spaces.
  {I}, {II}}, Ann. of Math. (2) \textbf{96} (1972), 98--128; ibid. (2) 96
  (1972), 129--173. \MR{0352946}

\bibitem[And73]{Ando-closedrange}
T.~Ando, \emph{Closed range theorems for convex sets and linear liftings},
  Pacific J. Math. \textbf{44} (1973), 393--410. \MR{0328546}

\bibitem[And74]{Andersen}
Tage~Bai Andersen, \emph{Linear extensions, projections, and split faces}, J.
  Functional Analysis \textbf{17} (1974), 161--173. \MR{0355560}

\bibitem[And75]{Ando-nonempty}
T.~Ando, \emph{A theorem on nonempty intersection of convex sets and its
  application}, J. Approximation Theory \textbf{13} (1975), 158--166,
  Collection of articles dedicated to G. G. Lorentz on the occasion of his
  sixty-fifth birthday. \MR{0385520}

\bibitem[Are84]{Arendt}
Wolfgang Arendt, \emph{Factorization by lattice homomorphisms}, Math. Z.
  \textbf{185} (1984), no.~4, 567--571. \MR{733776}

\bibitem[Bla16]{Blanco}
A.~Blanco, \emph{On the positive approximation property}, Positivity
  \textbf{20} (2016), no.~3, 719--742. \MR{3540521}

\bibitem[BM12]{BorelMathurin}
Laetitia Borel-Mathurin, \emph{Approximation properties and non-linear geometry
  of {B}anach spaces}, Houston J. Math. \textbf{38} (2012), no.~4, 1135--1148.
  \MR{3019026}

\bibitem[Bor33]{Borsuk}
K.~Borsuk, \emph{{\"U}ber {I}somorphie der {F}unktionalr\"aume}, Bull. Int.
  Acad. Polon. Sci. (1933), 1--10.

\bibitem[Car75]{Cartwright}
Donald~I. Cartwright, \emph{Extensions of positive operators between {B}anach
  lattices}, Mem. Amer. Math. Soc. \textbf{3} (1975), no.~164, iv+48.
  \MR{0383031}

\bibitem[CD17]{CD-ACE-respecting}
Javier~Alejandro Ch{\'a}vez-Dom{\'\i}nguez, \emph{An {A}ndo-{C}hoi-{E}ffros
  lifting theorem respecting subspaces}, arXiv preprint arXiv:1702.04774
  (2017).

\bibitem[CE77]{Choi-Effros}
Man~Duen Choi and Edward~G. Effros, \emph{Lifting problems and the cohomology
  of {$C\sp*$}-algebras}, Canad. J. Math. \textbf{29} (1977), no.~5,
  1092--1111. \MR{0463929}

\bibitem[Dea73]{Dean73}
David~W. Dean, \emph{The equation
  {$L(E,\,X\sp{\ast\ast})=L(E,\,X)\sp{\ast\ast}$} and the principle of local
  reflexivity}, Proc. Amer. Math. Soc. \textbf{40} (1973), 146--148.
  \MR{MR0324383 (48 \#2735)}

\bibitem[Dug51]{Dugundji}
J.~Dugundji, \emph{An extension of {T}ietze's theorem}, Pacific J. Math.
  \textbf{1} (1951), 353--367. \MR{0044116}

\bibitem[Fre74]{Fremlin}
D.~H. Fremlin, \emph{Tensor products of {B}anach lattices}, Math. Ann.
  \textbf{211} (1974), 87--106. \MR{0367620}

\bibitem[GO14]{Godefroy-Ozawa}
Gilles Godefroy and Narutaka Ozawa, \emph{Free {B}anach spaces and the
  approximation properties}, Proc. Amer. Math. Soc. \textbf{142} (2014), no.~5,
  1681--1687. \MR{3168474}

\bibitem[God15a]{Godefroy}
Gilles Godefroy, \emph{Extensions of {L}ipschitz functions and {G}rothendieck's
  bounded approximation property}, North-West. Eur. J. Math. \textbf{1} (2015),
  1--6. \MR{3417417}

\bibitem[God15b]{Godefroy-survey}
\bysame, \emph{A survey on {L}ipschitz-free {B}anach spaces}, Comment. Math.
  \textbf{55} (2015), no.~2, 89--118. \MR{3518958}

\bibitem[Hay77]{Haydon}
Richard Haydon, \emph{Injective {B}anach lattices}, Math. Z. \textbf{156}
  (1977), no.~1, 19--47. \MR{0473776}

\bibitem[HWW93]{Harmand-Werner-Werner}
P.~Harmand, D.~Werner, and W.~Werner, \emph{{$M$}-ideals in {B}anach spaces and
  {B}anach algebras}, Lecture Notes in Mathematics, vol. 1547, Springer-Verlag,
  Berlin, 1993. \MR{1238713}

\bibitem[Lab04]{Labuschagne}
C.~C.~A. Labuschagne, \emph{Riesz reasonable cross norms on tensor products of
  {B}anach lattices}, Quaest. Math. \textbf{27} (2004), no.~3, 243--266.
  \MR{2109665}

\bibitem[LO11]{Lissitsin-Oja}
Aleksei Lissitsin and Eve Oja, \emph{The convex approximation property of
  {B}anach spaces}, J. Math. Anal. Appl. \textbf{379} (2011), no.~2, 616--626.
  \MR{2784345}

\bibitem[MN91]{Meyer-Nieberg}
Peter Meyer-Nieberg, \emph{Banach lattices}, Universitext, Springer-Verlag,
  Berlin, 1991. \MR{1128093}

\bibitem[MP67]{Michael-Pelczynski}
E.~Michael and A.~Pe{\l}czy{\'n}ski, \emph{A linear extension theorem},
  Illinois J. Math. \textbf{11} (1967), 563--579. \MR{0217582}

\bibitem[Sch74]{Schaefer}
Helmut~H. Schaefer, \emph{Banach lattices and positive operators},
  Springer-Verlag, New York-Heidelberg, 1974, Die Grundlehren der
  mathematischen Wissenschaften, Band 215. \MR{0423039}

\bibitem[Ves73]{Vestertrom}
J\o{}rgen Vesterstr\o{}m, \emph{Positive linear extension operators for spaces
  of affine functions}, Israel J. Math. \textbf{16} (1973), 203--211.
  \MR{0343005}

\end{thebibliography}
\end{document}